\documentclass[a4paper]{amsart}

\usepackage{latexsym}
\usepackage{ifthen}
\usepackage[leqno]{amsmath}
\usepackage{enumerate}
\usepackage{hyphenat}
\usepackage{amstext,amsbsy,amsopn,amsthm,amsgen,amsfonts,amscd,amsxtra,upref}
\usepackage{mathrsfs}
\usepackage{euscript}
\usepackage{amssymb}
\usepackage{calc}
\usepackage[OT4]{fontenc}

\makeatletter\@namedef{subjclassname@2010}{\textup{2010} Mathematics Subject Classification}
 \makeatother

\swapnumbers
\theoremstyle{plain}
\newtheorem{Thm}{Theorem}[section]

\newtheorem{Cor}[Thm]{Corollary}
\newtheorem{Pro}[Thm]{Proposition}

\theoremstyle{definition}
\newtheorem{Def}[Thm]{Definition}
\newtheorem{Exm}[Thm]{Example}
\newtheorem{Exs}[Thm]{Examples}

\theoremstyle{remark}
\newtheorem{Rem}[Thm]{Remark}

\numberwithin{equation}{section}

\newcommand{\ITE}[3]{\ifthenelse{#1}{#2}{#3}}\newcommand{\ITEE}[4][]{\ITE{\equal{#2}{#3}}{#4}{#1}}
 \ITE{\isundefined{\texorpdfstring}}{\newcommand{\texorpdfstring}[2]{#1}}{}

\newenvironment{cor}[2][]{\ITEE[{\begin{Cor}[#1]}]{#1}{}{\begin{Cor}}\label{cor:#2}}{\end{Cor}}
\newenvironment{dfn}[2][]{\ITEE[{\begin{Def}[#1]}]{#1}{}{\begin{Def}}\label{def:#2}}{\end{Def}}
\newenvironment{exm}[2][]{\ITEE[{\begin{Exm}[#1]}]{#1}{}{\begin{Exm}}\label{exm:#2}}{\end{Exm}}
\newenvironment{exs}[2][]{\ITEE[{\begin{Exs}[#1]}]{#1}{}{\begin{Exs}}\label{exs:#2}}{\end{Exs}}
\newenvironment{pro}[2][]{\ITEE[{\begin{Pro}[#1]}]{#1}{}{\begin{Pro}}\label{pro:#2}}{\end{Pro}}
\newenvironment{rem}[2][]{\ITEE[{\begin{Rem}[#1]}]{#1}{}{\begin{Rem}}\label{rem:#2}}{\end{Rem}}
\newenvironment{thm}[2][]{\ITEE[{\begin{Thm}[#1]}]{#1}{}{\begin{Thm}}\label{thm:#2}}{\end{Thm}}

\newcommand{\COR}[2][!]{\ITEE{#1}{!}{Corollary~}\ITEE{#1}{s}{Corollaries~}\textup{\ref{cor:#2}}}
\newcommand{\DEF}[2][!]{\ITEE{#1}{!}{Definition~}\ITEE{#1}{s}{Definitions~}\textup{\ref{def:#2}}}
\newcommand{\EXS}[2][!]{\ITEE{#1}{!}{Examples~}\ITEE{#1}{1}{Example~}\textup{\ref{exs:#2}}}
\newcommand{\PRO}[2][!]{\ITEE{#1}{!}{Proposition~}\ITEE{#1}{s}{Propositions~}\textup{\ref{pro:#2}}}
\newcommand{\THM}[2][!]{\ITEE{#1}{!}{Theorem~}\ITEE{#1}{s}{Theorems~}\textup{\ref{thm:#2}}}

\newcommand{\RRR}{\mathbb{R}}
\newcommand{\DdD}{\EuScript{D}}
\newcommand{\FfF}{\EuScript{F}}
\newcommand{\NnN}{\EuScript{N}}
\newcommand{\RrR}{\EuScript{R}}
\newcommand{\ffF}{\mathscr{F}}
\newcommand{\ggG}{\mathscr{G}}
\newcommand{\llL}{\mathscr{L}}
\newcommand{\ssS}{\mathscr{S}}
\newcommand{\ttT}{\mathscr{T}}
\newcommand{\zzZ}{\mathscr{Z}}

\newcommand{\dd}{\colon}
\newcommand{\df}{\stackrel{\textup{def}}{=}}
\newcommand{\epsi}{\varepsilon}
\newcommand{\geqsl}{\geqslant}
\newcommand{\leqsl}{\leqslant}
\newcommand{\scalar}[2]{\left\langle#1,#2\right\rangle}
\newcommand{\scalarr}{\langle\cdot,\mathrm{-}\rangle}
\newcommand{\unit}[1]{#1^{\sqcup1}}
\newcommand{\varempty}{\varnothing}

\newcommand{\dist}{\operatorname{dist}}
\newcommand{\lin}{\operatorname{lin}}

\newcommand{\tfcae}{the following conditions are equivalent:}

\begin{document}

\title[Applications of amenable semigroups in operator theory]
 {Applications of amenable semigroups\\in operator theory}
\author[P.\ Niemiec]{Piotr Niemiec}
\address{P.\ Niemiec\\Instytut Matematyki\\
  Wydzia\l{} Matematyki i~Informatyki\\Uniwersytet Jagiello\'{n}ski\\
  ul.\ \L{}ojasiewicza 6\\30-348 Krak\'{o}w\\Poland}
\email{piotr.niemiec@uj.edu.pl}
\thanks{The first author is supported by the NCN (National Science Center in
 Poland), decision No. DEC-2013/11/B/ST1/03613.}
\author[P. W\'{o}jcik]{Pawe\l{} W\'{o}jcik}
\address{P. W\'{o}jcik\\Instytut Matematyki\\Uniwersytet Pedagogiczny\\%
 ul.\ Podchor\k{a}\.{z}ych 2\\{}30--084 Krak\'{o}w\\{}Poland}
\email{pwojcik@up.krakow.pl}

\begin{abstract}
The paper deals with continuous homomorphisms \(\ssS \ni s \mapsto T_s \in
\llL(E)\) of amenable semigroups \(\ssS\) into the algebra \(\llL(E)\) of all
bounded linear operators on a Banach space \(E\). For a closed linear subspace
\(F\) of \(E\), sufficient conditions are given under which there exists
a projection \(P \in \llL(E)\) onto \(F\) that commutes with all \(T_s\). And
when \(E\) is a Hilbert space, sufficient conditions are given for the existence
of an invertible operator \(L \in \llL(E)\) such that all \(L T_s L^{-1}\) are
isometries. Also certain results on extending intertwining operators, renorming
as well as on operators on hereditarily indecomposable Banach spaces are
offered.
\end{abstract}
\maketitle

\section{Introduction}

A (semi)group \(\ssS\) is \textit{amenable} if there exists a (so-called)
\textit{invariant mean} on a suitably chosen vector space of (certain) bounded
real-valued functions defined on \(\ssS\). This mean is a positive linear
functional invariant under (left, right or both left and right) translations of
the semigroup. The so-called \textit{Banach limits} on \(\ell_{\infty}\)
(see e.g.\ II.4.22 on page 73 in \cite{d-s}) are the most classical examples
of invariant means on a semigroup that is not a group. The above (sketch of a)
definition of an amenable semigroup deals with an intrinsic property of
the semigroup. However, the amenability is equivalent to a strong (and powerful)
property of a fixed point. Namely, a semigroup \(\ssS\) is amenable iff for
every affine action (with certain additional properties related to continuity
and corresponding to the version of amenability we deal with) \(K \ni x \mapsto
\phi_s(x) \in K\ (s \in \ssS)\) of \(\ssS\) on a compact convex non-empty subset
\(K\) of a locally convex topological vector space there is a point \(a \in K\)
such that \(\phi_s(a) = a\) for all \(s \in \ssS\). This property makes
amenability strongly applicable and thus amenable structures (which include
locally compact topological groups, abstract semigroups, extremely amenable
Polish groups, Banach algebras, \(C^*\)-algebras) are still widely investigated
and prospected. The simplest criterion for a topological group (or a semigroup)
to be amenable is the commutativity of its binary action. However, there are
a number of amenable groups (in particular, all compact topological) that are
non-abelian. The reader interested in the classical notion of amenability (for
topological groups or semigroups) is referred to, e.g., \cite{gre}, \cite{pat},
\cite{pie} or Chapter~1 in \cite{run}.\par
In this paper we propose two versions of amenability (see \DEF{amen} in
Section~3 below). Because of applications offered in this paper, we focus only
on right amenability. Our considerations comprise all semigroups equipped with
arbitrary topologies, that form a much wider class than topological semigroups.
This is motivated by a growing interest in semitopological semigroups (or Ellis
semigroups; these are related to ultrafilter techniques) in dynamical systems
and ergodic theory (see, e.g., Chapter~6 in \cite{aki}). We offer here certain
results on: the existence of projections commuting with a collection of bounded
operators on a Banach space (\PRO[s]{proj} and \ref{pro:fin}, \THM[s]{bdd} and
\ref{thm:fin} and \COR{proj}); extending intertwining operators (\THM{inter} and
\COR{bdd}); renorming the space in a way that all operators from a given family
become isometric (\THM{renorm} and \COR{enlarge}); as well as joint similarity
to isometries in a Hilbert space (\PRO{gen} and \COR{iso}). For clarity and
simplicity, below we formulate main of these results in the case of abelian
familes of operators:

\begin{thm}{m1}
Let \(\ttT \subset \llL(E)\) be an abelian multiplicative semigroup of operators
on a Banach space \(E\), and \(F\) be a closed linear subspace of \(E\) such
that:
\begin{enumerate}[\upshape(1)]
\item for any \(T \in \ttT\), \(T(F) = F\) and the restriction
 \(T\restriction{F}\) of \(T\) to \(F\) is an isomorphism;
\item \(F\) is a dual Banach space; that is, \(F\) is linearly isometric to
 the dual space of some Banach space \textup{(}and thus the closed unit ball
 \(B_F\) of \(F\) is compact with respect to certain locally convex topology,
 called further weak*, of \(F\)\textup{)};
\item for any \(T \in \ttT\), \(T\restriction{F}\) is weak* continuous on
 \(B_F\);
\item \(\sup_{T \in \ttT} (\|T\| \cdot \|(T\restriction{F})^{-1}\|) < \infty\).
\end{enumerate}
Then there exists a projection \(P \in \llL(E)\) onto \(F\) that commutes with
all \(T \in \ttT\) iff \(F\) is complemented in \(E\). More specifically, if
\(F\) is complemented in \(E\), then there exists a projection \(Q \in \llL(E)\)
onto \(F\) that commutes with all operators in \(\ttT\) and satisfies:
\begin{itemize}
\item \(\|Q\| \leqsl \sup_{T \in \ttT} (\|T\| \cdot \|(T\restriction{F})^{-1}\|)
 \cdot \lambda(E,F)\), where \(\lambda(E,F) = \inf \{\|P\|\dd\ P \in \llL(E)
 \textup{ a projection onto } F\}\),
\item \(Q\) is minimal; that is, \(\|Q\| \leqsl \|P\|\) for any other projection
 \(P\) onto \(F\) that commutes with all operators in \(\ttT\).
\end{itemize}
\end{thm}

A special case (for compact groups) of the generalisation of the above result
(formulated in \PRO{proj}) is due to Rudin \cite{rud}.

\begin{thm}{m2}
Let \(\ssS \ni s \mapsto A_s \in \llL(X)\) and \(\ssS \ni s \mapsto B_s \in
\llL(Y)\) be two homomorhisms, with \(B_s\) invertible for all \(s \in \ssS\),
of an abelian semigroup \(\ssS\), where \(X\) is a Banach space and \(Y\) is
a dual Banach space. Further, let \(T_0\dd E \to Y\) be a bounded linear
operator from a closed linear subspace \(E\) of \(X\) into \(Y\) such that for
any \(s \in \ssS\), \(A_s(E) \subset E\) and \(B_s T_0 = T_0
A_s\restriction{E}\). Then \tfcae
\begin{enumerate}[\upshape(i)]
\item \(T_0\) extends to a bounded linear operator \(T\dd X \to Y\) such that
 \(B_s T = T A_s\) for all \(s \in \ssS\);
\item there exists a bounded linear operator \(T'\dd X \to Y\) that extends
 \(T_0\) and satisfies \(\sup_{s\in\ssS} \|B_s^{-1} T' A_s\| < \infty\).
\end{enumerate}
\end{thm}

\begin{thm}{m3}
Let \(\ttT \subset \llL(E)\) be a bounded abelian multiplicative semigroup of
operators on a Banach space \(E\), and \(\NnN(\ttT) \df \{x \in E\dd\ Tx = x\ (T
\in \ttT)\}\) and \(\RrR(\ttT) \df \overline{\lin} \{Tx - x\dd\ T \in \ttT,\ x
\in E\}\). Then:
\begin{itemize}
\item \(\NnN(\ttT) \cap \RrR(\ttT) = \{0\}\),
\item the subspace
 \begin{equation}\label{eqn:dec}
 \DdD(\ttT) \df \NnN(\ttT) + \RrR(\ttT)
 \end{equation}
 is closed in \(E\), and
\item the projection \(P\dd \DdD(\ttT) \to \NnN(\ttT)\) induced by
 the decomposition \eqref{eqn:dec} has norm not greater than
 \(\sup_{T\in\ttT} \|T\|\).
\end{itemize}
\end{thm}

\begin{thm}{m4}
Let \(\ttT \subset \llL(H)\) be a bounded abelian multiplicative semigroup of
operators on a Hilbert space \(H\). Suppose there are two positive real
constants \(m\) and \(M\) such that for any \(T \in \ttT\) and \(x \in H\),
\[m \|x\| \leqsl \|T x\| \leqsl M \|x\|.\]
Then there exists an invertible positive operator \(A \in \llL(H)\) such that
\(m I \leqsl A \leqsl M I\) and \(A T A^{-1}\) is an isometry for each \(T \in
\ttT\).
\end{thm}

The above theorem generalises a classical result of Sz.-Nagy \cite{SzN} on
bounded abelian groups of Hilbert space operators.\par
The paper is organised as follows. In the next section we introduce basic
notions and fix the notation. Section~3 is a preliminary part on two variants of
amenability we will deal with further. The fourth, main part is devoted to
applications of amenablility defined in the previous one. We prove and formulate
there all results dealing with amenable semigroups and operators in Banach
spaces. This section is concluded by a proof of \THM[s]{m1}--\ref{thm:m3}.
In the last, fifth section we give more information on the two kinds of
amenability introduced in Section~3.

\section{Notation and terminology}

In this paper all semigroups are non-empty and Banach spaces are real or
complex. We use multiplicative notation to denote binary actions of abstract
semigroups (excluding vector spaces). The abbreviation \textit{SGT} means
a \textit{semigroup with topology}. So, a statement \textit{\(\ssS\) is an SGT}
means that \(\ssS\) is both a semigroup and a (possibly non-Hausdorff)
topological space, and \textbf{no compatibility axiom (between these two
structures) is assumed}. In particular, in an SGT the binary action can be
separately discontinuous. Although SGT's need not be Hausdorff, for us
topological groups, by definition, are. An important example of non-topological
SGT's are Ellis semigroups in which the binary action is continuous only in
the left variable (according to the definition given in Chapter~6 of
\cite{aki}). A classical example of compact Ellis semigroups that widely appear
in functional analysis are the closed unit balls (considered as multiplicative
semigroups with the weak* operator topologies) of the Banach spaces of the form
\(\llL(X^*)\) where \(X^*\) is the dual Banach space of a Banach space \(X\)
(only when \(X\) is reflexive, the multiplication is separately continuous in
these SGT's).\par
An SGT does not need to have a neutral element, and if it does have, it is said
to be \textit{unital} (otherwise it is called \textit{non-unital}). Similarly,
a homomorphism between two unital SGT's is called \textit{unital} if it sends
the unit of the source semigroup to the unit of the target.\par
By a \textit{right action} of a semigroup \(\ssS\) on a set \(X\) we mean any
function \(X \times \ssS \ni (x,s) \mapsto \phi_s(x) = x . s \in X\) such that
for all \(s, t \in \ssS\) and \(x \in X\), \((x . s) . t = x . (st)\). If, in
addition,
\begin{itemize}
\item \(\ssS\) is unital and \(x . 1 = x\) for any \(x \in X\), the action is
 called \textit{unital};
\item \(X\) is a convex set (in a real vector space) and \(\phi_s\) is
 affine---that is, \[((1-\alpha) x + \alpha y).s = (1-\alpha) (x.s) + \alpha
 (y.s)\] for any \(x, y \in X\) and \(\alpha \in [0,1]\), the action is called
 \textit{affine}.
\end{itemize}
A point \(a \in X\) is said to be a \textit{fixed point} for the action if \(a.s
= a\) for any \(s \in \ssS\).\par
When \((x,s) \mapsto x . s\) is a right action of \(\ssS\) on a set \(X\) and
\(u\dd X \to Y\) is an arbitrary function, for any \(s \in \ssS\), \(u_s\dd X
\to Y\) is defined by \(u_s(t) = u(ts)\). This, in particular, applies to
the natural right action of \(\ssS\) on itself.\par
For a topological space \(X\), \(C_b(X)\) stands for the algebra of all bounded
continuous real-valued functions on \(X\), equipped with the sup-norm. If \(X\)
is compact, we write \(C(X)\) in place of \(C_b(X)\).\par
For any Banach space \(E\), \(\llL(E)\) denotes the algebra of all bounded
linear operators from \(E\) into \(E\). When \(F\) is a closed linear subspace
of \(E\), a statement ``\textit{\(P\dd E \to F\) is a projection}'' means that
\(P\) is a bounded linear operator such that \(Pf = f\) for any \(f \in F\).
If \(F\) is complemented in \(E\), \(\lambda(E,F)\) is used to denote
the quantity specified in \THM{m1}. A function \(\ssS \ni s \mapsto T_s \in
\llL(E)\) is a homomorphism if \(T_{st} = T_s T_t\) for all \(s, t \in \ssS\).

\section{Two kinds of amenability}

In the literature there are two classes of SGT's in which amenability is well
studied---these are the classes of locally compact topological groups (consult,
e.g., \cite{gre}, \cite{pat}, \cite{pie} or Chapter~1 in \cite{run}) and of
abstract semigroups (that is, semigroups without topologies or, equivalently,
discrete semigroups; see, e.g., \cite{pat} or Section~1 in \cite{gre}).
Amenability was also generalised to Banach algebras (see, e.g., Chapter~2 in
\cite{run} or \cite{jhp} and references therein). Amenability in
\(C^*\)-algebras also has a special interest (consult, e.g., \cite{lin}).\par
A topic closely related to amenable groups are \textit{extremely amenable Polish
groups}, which are now intensively studied. According to our \DEF{amen} (see
below), all such groups are weakly right amenable. The literature on extremely
amenable Polish groups is still growing up and we mention only a few of them:
\cite{g-m}, \cite{kpt}, \cite{r-s}, \cite{f-s}, \cite{sab} and references
therein or \cite{m-t}.\par
Because of further applications, we will deal with two versions of right
amenability. Of course, in a similar manner one can introduce their counterparts
for left amenability.\par
From now on, let \(\ssS\) be an SGT and \(C_r(\ssS)\) consist of all functions
\(f \in C_b(\ssS)\) such that \(f_s\) is continuous for any \(s \in \ssS\).\par
Let \(\FfF\) be a linear subspace of \(C_r(\ssS)\). A \textit{mean} on \(\FfF\)
is a linear functional \(\phi\dd \FfF \to \RRR\) such that \(\inf f(\ssS) \leqsl
\phi(f) \leqsl \sup f(\ssS)\) for any \(f \in \FfF\). (If \(\FfF\) contains
a function \(j\) constantly equal to \(1\), a linear functional \(\phi\) on
\(\FfF\) is a mean iff \(\|\phi\| = \phi(j) = 1\).) The space \(\FfF\) is said
to be \textit{right invariant} if \(u_s \in \FfF\) for any \(u \in \FfF\) and
\(s \in \ssS\). (It is easily seen that \(C_r(\ssS)\) is right invariant.)\par
A mean \(\phi\) on a right invariant linear subspace \(\FfF\) of \(C_r(\ssS)\)
is \textit{right invariant} if \(\phi(f_s) = \phi(f)\) for any \(f \in \FfF\)
and \(s \in \ssS\).\par
Right amenability deals with right invariant means and asks about their
existence on suitably chosen right invariant linear spaces \(\FfF\). The bigger
this space \(\FfF\) is, the better version of amenability we get. Below we
propose two of them.

\begin{dfn}{amen}
An SGT \(\ssS\) is said to be \textit{weakly right amenable} if there is a right
invariant mean on
\begin{multline*}
C_{norm}(\ssS) \df \{f \in C_r(\ssS)\dd\ s \mapsto f_s \textup{\quad and\quad}
s \mapsto f_{st}\textup{\quad are continuous}\\\textup{in the norm topology of
\(C_r(\ssS)\) for all } t \in \ssS\}.
\end{multline*}
\(\ssS\) is said to be \textit{strongly right amenable} if there is a right
invariant mean on
\begin{multline*}
C_{weak}(\ssS) \df \{f \in C_r(\ssS)\dd\ s \mapsto f_s \textup{\quad and\quad}
s \mapsto f_{st}\textup{\quad are continuous}\\\textup{in the weak topology of
\(C_r(\ssS)\) for all } t \in \ssS\}.
\end{multline*}
Since \(C_{norm}(\ssS) \subset C_{weak}(\ssS)\), a strongly right amenable SGT
is weakly right amenable.
\end{dfn}

The following result, proved by Day \cite{day,dy2} (see Theorem~3 therein) for
amenable abstract semigroups, is a powerful tool that makes amenable SGT's
useful.

\begin{thm}{fpp}
For an SGT \(\ssS\) \tfcae
\begin{enumerate}[\upshape(SG1)]
\item \(\ssS\) is weakly \textup{[}resp.\ strongly\textup{]} right amenable;
\item every jointly \textup{[}resp.\ separately\textup{]} continuous affine
 right action of \(\ssS\) on a non-empty compact convex set in a Hausdorff
 locally convex topological vector space has a fixed point.
\end{enumerate}
If, in addition, \(\ssS\) is unital, then the above conditions are equivalent
to:
\begin{itemize}
\item[\textup{(SG2')}] every jointly \textup{[}resp.\ separately\textup{]}
 continuous unital affine right action of \(\ssS\) on a non-empty compact convex
 set in a Hausdorff locally convex topological vector space has a fixed point.
\end{itemize}
\end{thm}

Here we postpone the proof of the above theorem to Section~5 where we give its
sketch for the sake of completeness.

\begin{exs}{amen}
\begin{enumerate}[(A)]
\item Let \(\ggG\) be a compact topological group. Then \(C_{weak}(\ggG) =
 C_{norm}(\ggG) = C(\ggG)\) and \(\ggG\) is strongly right amenable.
 The functional induced by the Haar measure of \(\ggG\) is a \underbar{unique}
 invariant mean on \(C(\ggG)\).
\item Every abelian semigroup \(\ssS\) equipped with the discrete topology is
 strongly right amenable and thus \(\ssS\) is such when equipped with totally
 arbitrary topology. This is a consequence of a well-known Markov-Kakutani fixed
 point theorem \cite{mar}, \cite{kak}. The fact that abelian SGT's are strongly
 right amenable will be used in the proof of all results formulated in
 Section~1.
\item When the topology of an SGT \(\ssS\) is discrete, \(C_{weak}(\ssS) =
 C_{norm}(\ssS) =\) the algebra of all bounded real-valued functions on \(\ssS\)
 and thus \(\ssS\) is strongly right amenable iff it is weakly right amenable,
 iff it is right amenable as an abstract semigroup (considered without topology)
 \cite{run}.
\item Let \(\ggG\) be a locally compact topological group and \(UC(\ggG)\) stand
 for the algebra of all those functions \(f \in C_b(\ggG)\) that are uniformly
 continuous (cf.\ \cite{run}). It can be easily shown that \(UC(\ggG) \subset
 C_{norm}(\ggG) \subset C_{weak}(\ggG) \subset C_b(\ggG)\) and thus, thanks to
 Theorem~1.1.9 in \cite{run}, \(\ggG\) is strongly right amenable iff it is
 weakly right amenable, iff \(\ggG\) is amenable in a classical sense (that is,
 if there is an invariant mean on \(L^{\infty}(\ggG)\)).
\item It is well known that the free (non-abelian) group \(\ffF_2\) in two
 generators is non-amenable as an abstract group---or, equivalently, is not
 weakly right amenable when equipped with the discrete topology. It is also well
 known that there are two orthogonal matrices \(U, V \in O_3\) that generate
 a group \(\ggG\) isomorphic to \(\ffF_2\). Since the closure \(\bar{\ggG}\) of
 \(\ggG\) (in \(O_3\)) is compact, \(\bar{\ggG}\) is (strongly and thus) weakly
 right amenable. In Section~5 we will show that a dense subgroup of a weakly
 right amenable topological group is weakly right amenable as well (see
 \THM{dense}). We conclude that \(\ffF_2\) admits a separable metrizable
 topology that makes \(\ffF_2\) a weakly right amenable topological group.
\end{enumerate}
\end{exs}

More on right amenability the reader can find in Section~5.

\section{Applications}

In this section we show how right amenability can be applied in operator theory.
Although our first result has nothing in common with amenability, the method of
its proof was one of our motivations.

\begin{pro}{uniq}
Let a closed linear subspace \(Y\) of a Banach space \(X\) and \(T \in \llL(X)\)
be such that there is a unique minimal projection \(P\) of \(X\) onto \(Y\),
\(T(Y) = Y\), \(T\restriction{Y}\) is an isometry and \(\|T\| = 1\). Then \(TP =
PT\).
\end{pro}
\begin{proof}
Observe that \(Q \df (T\restriction{Y})^{-1} P T\) is a projection onto \(Y\)
such that \(\|Q\| \leqsl \|P\|\). Since \(P\) is unique minimal, we have \(Q =
P\) and hence \(TP = PT\).
\end{proof}

The idea of the above proof will be applied in the proof of the following

\begin{pro}{proj}
Let \(\Phi\dd \ssS \ni s \mapsto T_s \in \llL(X)\) be a homomorphism of
a strongly \textup{[}resp.\ weakly\textup{]} right amenable SGT \(\ssS\) into
\(\llL(X)\) where \(X\) is a Banach space. Suppose \(Y\) is a closed linear
subspace of \(X\) and:
\begin{enumerate}[\upshape(p1)]\addtocounter{enumi}{-1}
\item \(\Phi\) is continuous in the strong operator topology \textup{[}resp.\
 operator norm topology\textup{]} of \(\llL(X)\);
\item for any \(s \in \ssS\), \(T_s(Y) = Y\) and the restriction
 \(T_s\restriction{Y}\) of \(T_s\) to \(Y\) is an isomorphism;
\item \(Y\) is a dual Banach space; that is, \(Y\) is linearly isometric to
 the dual space of some Banach space \textup{(}and thus the closed unit ball
 \(B_Y\) of \(Y\) is compact with respect to certain locally convex topology,
 called further weak*, of \(Y\)\textup{)};
\item for any \(s \in \ssS\), \(T_s\restriction{Y}\) is weak* continuous on
 \(B_Y\);
\item the function \(\ssS \ni s \mapsto \|(T_s\restriction{Y})^{-1}\| \in \RRR\)
 is locally bounded.
\end{enumerate}
Then \tfcae
\begin{enumerate}[\upshape(i)]
\item there exists a projection \(P\dd X \to Y\) such that \(PT_s = T_s P\) for
 all \(s \in \ssS\);
\item there is a projection \(Q\dd X \to Y\) such that
 \begin{equation}\label{eqn:bdd}
 \sup_{s \in \ssS} \|(T_s\restriction{Y})^{-1}QT_s\| < \infty.
 \end{equation}
\end{enumerate}
Moreover, if \textup{(ii)} holds, then there exists a projection \(P_0\dd X \to
Y\) such that \(P_0\) commutes with all operators \(T_s\) and \(\|P_0\| \leqsl
\|P\|\) for any other projection \(P\dd X \to Y\) commuting with all \(T_s\).
\end{pro}

The part concerning a `minimal' projection \(P_0\) generalises a classical
theorem due to Cheney and Morris \cite{c-m} on the existence of minimal
projections for complemented subspaces that are dual Banach spaces. It is worth
noting here that the above condition (p4) is automatically fulfilled (thanks to
(p0)) in the version of this proposition for weakly right amenable SGT's.

\begin{proof}[Proof of \PRO{proj}]
First of all, (p3) combined with Kre\u{\i}n-Smulian theorem (that characterises
weak* closed convex sets) implies that
\begin{itemize}
\item[(p3')] for any \(s \in \ssS\), \(T_s\restriction{Y}\) is a dual operator
 and \((T_s\restriction{Y})^{-1}\) is weak* continuous (on the whole \(Y\)).
\end{itemize}
For a projection \(P\dd X \to Y\) and \(s \in \ssS\) we write \(P.s\) to denote
\((T_s\restriction{Y})^{-1}PT_s\). Observe that \(P.s\dd X \to Y\) is again
a projection and that
\begin{equation}\label{eqn:aux1}
(P,s) \mapsto P.s
\end{equation}
is an affine right action of \(\ssS\) on the set of all projection from \(X\)
onto \(Y\).\par
Assume \(Q\) is as specified in (ii). Put \(M \df \sup_{s\in\ssS} \|Q.s\| <
\infty\). Let the set \(\Delta\) of all projections \(P\dd X \to Y\) with
\(\|P\| \leqsl M\) and \(\sup_{s\in\ssS} \|P.s\| \leqsl M\) be equipped with
the topology of pointwise weak* convergence (we call it \textsl{weak* operator
topology}). It is easily seen that \(\Delta\) is convex, \(Q.s \in \Delta\) and
\(P.s \in \Delta\) for any \(P \in \Delta\) and \(s \in \ssS\). Moreover, it
follows from (p3') that \(\Delta\) is compact. So, by \THM{fpp}, it suffices
to show that the action \eqref{eqn:aux1} is separately [resp.\ jointly]
continuous.\par
It follows from (p3') that \(\Delta \ni P \mapsto P.s \in \Delta\) is continuous
for each \(s \in \ssS\). Further, if \(x\) is an arbitrary vector of \(X\) and
\((s_{\sigma})_{\sigma\in\Sigma}\) is a net in \(\ssS\) convergent to \(s \in
\ssS\), then \(\lim_{\sigma\in\Sigma} T_{s_{\sigma}} x = T_s x\) and hence
(thanks to (p4)) the nets \((T_{s_{\sigma}} x)_{\sigma\in\Sigma}\) and
\((\|(T_{s_{\sigma}}\restriction{Y})^{-1}\|)_{\sigma\in\Sigma}\) are eventually
bounded. This implies that \((T_{s_{\sigma}}\restriction{Y})^{-1}\) converge in
the strong operator topology to \((T_s\restriction{Y})^{-1}\) and
\(\lim_{\sigma\in\Sigma} (P.{s_{\sigma}}) x = (P.s) x\) for every \(P \in
\Delta\). So, the action is separately continuous. Finally, assume that, in
addition, \(\ssS\) is weakly right amenable, and then---by (p0)---we get
\begin{equation}\label{eqn:aux2}
\lim_{\sigma\in\Sigma} \|T_{s_{\sigma}} - T_s\| = 0
\end{equation}
as well as \(\lim_{\sigma\in\Sigma} \|(T_{s_{\sigma}}\restriction{Y})^{-1}
- (T_s\restriction{Y})^{-1}\| = 0\). Now if \((P_{\sigma})_{\sigma\in\Sigma}\)
is a net in \(\Delta\) convergent (in the topology of \(\Delta\)) to \(P \in
\Delta\), then the net \(((T_{s_{\sigma}}\restriction{Y})^{-1}
P_{\sigma})_{\sigma\in\Sigma}\) is bounded and converges in the weak* operator
topology to \((T_s\restriction{Y})^{-1} P\) (thanks to (p3') and the boundedness
of \(\Delta\)), which---combined with \eqref{eqn:aux2}---implies that
\(P_{\sigma}.s_{\sigma}\) converge to \(P.s\) in the topology of \(\Delta\).\par
To complete the proof, observe that the additional claim of the proposition
follows from the compactness of \(\Delta\) in the weak* operator topology, and
from the lower semicontinuity of the operator norm in this topology.
\end{proof}

\begin{rem}{weak}
\PRO{proj} in the version for weak right amenability can easily be generalised
(in a sense) to closed linear subspaces \(Y\) of \(X\) that are not dual Banach
spaces. More precisely, if all assumptions of that result (for weak right
amenability) hold, except for (p2) and (p3), and condition (ii) (therein) is
fulfilled, then there exists a projection \(P\dd X^{**} \to Y^{\perp\perp}\)
that commutes with all \(T_s^{**}\). Indeed, it suffices to observe that under
the settings described above, \underline{all} assumptions of \PRO{proj} are
satisfied for the homomorphism \(s \mapsto T_s^{**}\) (with \(X\) and \(Y\)
replaced by \(X^{**}\) and \(Y^{\perp\perp}\), respectively).
\end{rem}

\begin{dfn}{min}
When \(Y\) is a closed linear subspace of a Banach space \(X\) and \(\Phi\dd
\ssS \ni s \mapsto T_s \in \llL(X)\) is a homomorphism, a projection \(P_0\dd X
\to Y\) is called \textit{\(\Phi\)-minimal} if \(P_0\) commutes with all
\(T_s\), and \(\|P_0\| \leqsl \|P\|\) for any other projection \(P\dd X \to Y\)
that commutes with all \(T_s\).
\end{dfn}

When \(Y\) is a finite-dimensional subspace, \PRO{proj} can be strengthened
as follows.

\begin{pro}{fin}
Let \(\Phi\dd \ssS \ni s \mapsto T_s \in \llL(X)\) be a homomorphism of
a strongly \textup{[}resp.\ weakly\textup{]} right amenable SGT \(\ssS\) into
\(\llL(X)\) where \(X\) is a Banach space. Suppose \(Y\) is a finite-dimensional
linear subspace of \(X\) and:
\begin{enumerate}[\upshape(fp1)]\addtocounter{enumi}{-1}
\item \(\Phi\) is continuous in the weak operator topology \textup{[}resp.\
 strong operator topology\textup{]} of \(\llL(X)\);
\item for any \(s \in \ssS\), \(T_s(Y) = Y\).
\end{enumerate}
Then \tfcae
\begin{enumerate}[\upshape(i)]
\item there exists a projection \(P\dd X \to Y\) such that \(PT_s = T_s P\) for
 all \(s \in \ssS\);
\item there is a projection \(Q\dd X \to Y\) such that
 \[\sup_{s \in \ssS} \|(T_s\restriction{Y})^{-1}QT_s\| < \infty.\]
\end{enumerate}
Moreover, if \textup{(ii)} holds, then there exists a \(\Phi\)-minimal
projection \(P_0\dd X \to Y\).
\end{pro}

The proof is quite similar to the previous one and thus we skip it. (Use
the fact that under the assumptions of \PRO{fin}, both the functions
\[\ssS \ni s \mapsto T_s\restriction{Y} \in \llL(Y) \qquad \textup{and} \qquad
\ssS \ni s \mapsto (T_s\restriction{Y})^{-1} \in \llL(Y)\]
are continuous in the operator norm topology of \(\llL(Y)\).)

Consequences of \PRO[s]{proj} and \ref{pro:fin} follow. The following result
(with a different proof) for compact topological groups can be found in
\cite{rud} (see also Theorem~5.18 in \cite{rd2}) with a slightly different
settings.

\begin{thm}{bdd}
Let \(\Phi\dd \ssS \ni s \mapsto T_s \in \llL(X)\) be a homomorphism of
a strongly \textup{[}resp.\ weakly\textup{]} right amenable SGT \(\ssS\) into
\(\llL(X)\) where \(X\) is a Banach space. Suppose \(Y\) is a closed linear
subspace of \(X\) such that conditions \textup{(p0)--(p3)} of \PRO{proj} are
fulfilled and:
\begin{itemize}
\item[\upshape(p4')] the quantity \(m_Y(\Phi) \df \sup_{s\in\ssS} (\|T_s\| \cdot
 \|(T_s\restriction{Y})^{-1}\|)\) is finite.
\end{itemize}
Then there exists a \(\Phi\)-minimal projection \(P_0\dd X \to Y\) iff \(Y\) is
a complemented subspace of \(X\). Moreover, if \(Y\) is complemented, each
\(\Phi\)-minimal projection \(P\dd X \to Y\) satisfies
\begin{equation}\label{eqn:aux3}
\|P\| \leqsl m_Y(\Phi) \lambda(X,Y).
\end{equation}
\end{thm}
\begin{proof}
First of all, observe that (p4') implies (p4) (from \PRO{proj}): otherwise we
would have a net \((s_{\sigma})_{\sigma\in\Sigma} \subset \ssS\) convergent to
some \(s \in \ssS\) such that
\begin{equation}\label{eqn:aux7}
\lim_{\sigma\in\Sigma} \|(T_{s_{\sigma}}\restriction{Y})^{-1}\| = \infty.
\end{equation}
Then (p4') would imply that \(\lim_{\sigma\in\Sigma} \|T_{s_{\sigma}}\| = 0\)
and consequently \(T_s = 0\), and \(Y = \{0\}\) (by (p1)), which contradics
\eqref{eqn:aux7}.\par
Now thanks to \PRO{proj}, we only need to prove that a \(\Phi\)-minimal
projection \(P\dd X \to Y\) satisfies \eqref{eqn:aux3} provided \(Y\) is
complemented. By Cheney-Morris' theorem \cite{c-m}, there exists a projection
\(Q\dd X \to Y\) such that \(\|Q\| = \lambda(X,Y)\). Now repeat the proof of
\PRO{proj} with \(M \df m_Y(\Phi) \lambda(X,Y)\) to conclude the existence of
a projection \(P\dd X \to Y\) that commutes with all \(T_s\) and satisfies
\eqref{eqn:aux3}. The proof is complete.
\end{proof}

The proof of the next result goes similarly as above and thus is omitted.

\begin{thm}{fin}
Let \(\Phi\dd \ssS \ni s \mapsto T_s \in \llL(X)\) be a homomorphism of
a strongly \textup{[}resp.\ weakly\textup{]} right amenable SGT \(\ssS\) into
\(\llL(X)\) where \(X\) is a Banach space. Suppose \(Y\) is a finite-dimensional
linear subspace of \(X\) such that conditions \textup{(fp0)--(fp1) (}of
\PRO{proj}\textup{)} and \textup{(p4') (}of \THM{bdd}\textup{)} hold. Then there
exists a \(\Phi\)-minimal projection \(P\dd X \to Y\), and \(\|P\| \leqsl
m_Y(\Phi) \lambda(X,Y)\).
\end{thm}

For simplicity, for any homomorphism \(\Phi\dd \ssS \to \llL(X)\) (where \(X\)
is a Banach space) we denote by \(\NnN(\Phi)\) and \(\RrR(\Phi)\)
the intersection of the kernels and, respectively, the closed linear span of
the ranges of all operators of the form \(\Phi(s) - I\) where \(s \in \ssS\) and
\(I\) is the identity operator on \(X\). When \(\Phi\) is the identity function
(and \(\ssS \subset \llL(X)\)), we write \(\NnN(\ssS)\) and \(\RrR(\ssS)\)
instead of \(\NnN(\Phi)\) and \(\RrR(\Phi)\).

\begin{cor}{proj}
For an SGT \(\ssS\) \tfcae
\begin{enumerate}[\upshape(a)]
\item \(\ssS\) is strongly \textup{[}resp.\ weakly\textup{]} right amenable;
\item for any homomorphism \(\Phi\dd \ssS \to \llL(X)\) continuous in the weak
 \textup{[}resp.\ strong\textup{]} operator topology with finite quantity
 \(M(\Phi) \df \sup_{s\in\ssS} \|T_s\|\) one has:
 \begin{itemize}
 \item \(\NnN(\Phi) \cap \RrR(\Phi) = \{0\}\),
 \item the subspace
  \begin{equation}\label{eqn:deco}
  \DdD(\Phi) \df \NnN(\Phi) + \RrR(\Phi)
  \end{equation}
  is closed in \(X\), and
 \item the projection \(P\dd \DdD(\Phi) \to \NnN(\Phi)\) induced by
 the decomposition \eqref{eqn:deco} has norm not greater than \(M(\Phi)\).
 \end{itemize}
\end{enumerate}
\end{cor}
\begin{proof}
Assume (b) holds, put \(X = C_{norm}(\ssS)\) (resp.\ \(X = C_{weak}(\ssS)\)) and
define \(\Phi\dd \ssS \ni s \mapsto T_s \in \llL(X)\) by \(T_s f = f_s\). Note
that \(\Phi\) is a homomorphism continuous in the weak (resp.\ strong) operator
topology such that \(M(\Phi) = 1\). So, it follows from (b) that the projection
\(P\dd \DdD(\Phi) \to \NnN(\Phi)\) specified therein has norm not exceeding
\(1\). Observe that the function \(j\dd \ssS \to \RRR\) constantly equal to
\(1\) belongs to \(\NnN(\Phi)\) and thus \(1 = \|j\| = \|P(j+z)\| \leqsl
\|j+z\|\) for any \(z \in \RrR(\Phi)\). Equivalently, \(\dist(j,\RrR(\Phi)) =
1\). Now the Hahn-Banach theorem implies that there is a linear functional
\(\phi\dd X \to \RRR\) that vanishes on \(\RrR(\Phi)\) and satisfies \(\phi(j) =
1 = \|\phi\|\). This means that \(\phi\) is an invariant mean of \(X\) and we
are done.\par
Now assume (a) holds and let \(X\) and \(\Phi\dd \ssS \ni s \mapsto T_s \in
\llL(X)\) be as specified in (b). Fix a non-zero vector \(y \in \NnN(\Phi)\) and
consider the linear subspace \(Y\) generated by \(y\). Since \(Y\) is
one-dimensional, \(\lambda(X,Y) = 1\). Observe that conditions (fp0) and (fp1)
hold, and \(m_Y(\Phi) = M(\Phi)\). So, \THM{fin} gives us a \(\Phi\)-minimal
projection \(P_0\dd X \to Y\) with \(\|P_0\| \leqsl M(\Phi)\). For any \(s \in
\ssS\) and \(x \in X\) we have \[P_0(T_s x - x) = T_s(P_0 x) - P_0 x = 0,\]
because \(P_0 x \in Y \subset \NnN(\Phi)\). So, \(T_s x - x \in \ker(P_0)\) and
consequently \(\RrR(\Phi) \subset \ker(P_0)\). We conclude that for any \(z \in
\RrR(\Phi)\), \(\|y\| = \|P_0(y+z)\| \leqsl M(\Phi) \|y + z\|\). It follows from
the arbitrariness of \(y\) and \(z\) that \(\NnN(\Phi) \cap \RrR(\Phi) = \{0\}\)
and \(\|P\| \leqsl M(\Phi)\) (where \(P\) is as specified in (b)).
The closedness of \(\NnN(\Phi) + \RrR(\Phi)\) follows from the continuity of
\(P\) and the closedness of both \(\NnN(\Phi)\) and \(\RrR(\Phi)\).
\end{proof}

Recall that a Banach space \(X\) is HI (\textup{hereditary indecomposable}) if
for every projection \(P\dd Y \to Z\) where \(Z \subset Y \subset X\) the kernel
or the range of \(P\) is finite-dimensional.\par
As an immediate consequence of \COR{proj} (see item (B) in \EXS{amen}) we obtain

\begin{cor}{HI}
Let \(\ttT \subset \llL(E)\) be a bounded abelian multiplicative semigroup of
operators on a HI Banach space \(E\). Then either \(\NnN(\ttT)\) or
\(\RrR(\ttT)\) is finite-dimensional.
\end{cor}

Now we offer a result on extending intertwining operators that reads as follows.

\begin{thm}{inter}
Let \(\Phi\dd \ssS \ni s \mapsto A_s \in \llL(X)\) and \(\Psi\dd \ssS \ni s
\mapsto B_s \in \llL(Y)\) be two homomorphisms of a strongly \textup{[}resp.\
weakly\textup{]} right amenable SGT \(\ssS\) into \(\llL(X)\) and \(\llL(Y)\)
where \(X\) and \(Y\) are Banach spaces. Suppose \(T_0\dd E \to Y\) is a bounded
linear operator from a closed linear subspace \(E\) of \(X\) into \(Y\) and:
\begin{enumerate}[\upshape({i}1)]\addtocounter{enumi}{-1}
\item \(\Phi\) and \(\Psi\) are continuous in the strong operator topologies
 \textup{[}resp.\ operator norm topologies\textup{]} of \(\llL(X)\) and
 \(\llL(Y)\), respectively;
\item for any \(s \in \ssS\), \(A_s(E) \subset E\) and \(B_s\) is
 an isomorphism;
\item \(Y\) is a dual Banach space;
\item for any \(s \in \ssS\), \(B_s\) is weak* continuous on the closed unit
 ball \(B_Y\) of \(Y\);
\item the function \(\ssS \ni s \mapsto \|B_s^{-1}\| \in \RRR\) is locally
 bounded;
\item \(B_s T_0 = T_0 A_s\restriction{E}\) for any \(s \in \ssS\).
\end{enumerate}
Then \tfcae
\begin{enumerate}[\upshape(i)]
\item \(T_0\) extends to a bounded linear operator \(T\dd X \to Y\) that
 intertwines \(\Phi\) and \(\Psi\), that is, \(B_s T = T A_s\) for all \(s \in
 \ssS\);
\item there exists a bounded linear operator \(T'\dd X \to Y\) that extends
 \(T_0\) and satisfies \(\sup_{s \in \ssS} \|B_s^{-1} T' A_s\| < \infty\).
\end{enumerate}
\end{thm}
\begin{proof}
The proof is similar to that of \PRO{proj} and thus we only give its sketch. For
any bounded linear operator \(L\dd X \to Y\) and \(s \in \ssS\) we denote by
\(L.s\) the operator \(B_s^{-1} L A_s\). This defines an affine right action of
\(\ssS\). Next, assume \(T'\) is as specified in (ii) and set \(M \df
\sup_{s\in\ssS} \|T'.s\|\). Equip the set \(\Delta\) of all linear extensions
\(T\dd X \to Y\) of \(T_0\) such that \(\|T\| \leqsl M\) and \(\sup_{s\in\ssS}
\|T.s\| \leqsl M\) with the weak* operator topology. Now it suffices to repeat
the arguments presented in the proof of \PRO{proj} to conclude that the action
defined above is suitably continuous on \(\Delta \times \ssS\) and thus has
a fixed point \(T\) in \(\Delta\). Then automatically \(B_s T = T A_s\) for any
\(s \in \ssS\) and we are done.
\end{proof}

The following result is a special case of \THM{inter} and we skip its proof.

\begin{cor}{bdd}
Suppose all assumptions of \THM{inter} \textup{(}on \(\Phi\), \(\Psi\), \(Y\)
and \(T_0\)\textup{)} hold. If, in addition, \(\sup_{s\in\ssS} (\|B_s^{-1}\|
\cdot \|A_s\|) < \infty\), then \(T_0\) extends to a bounded linear operator
\(T\dd X \to Y\) intertwining \(\Phi\) and \(\Psi\) iff \(T_0\) extends to
a bounded linear operator from \(X\) into \(Y\).
\end{cor}

\begin{exm}{counter}
As the following simple example shows, the assumption in \THM{bdd} (and hence
also in \PRO{proj}) that \(T_s(Y) = Y\) for all \(s \in \ssS\) (see (p1)) cannot
be dropped in general. Let \(X\) be a separable Hilbert space with
an orthonormal basis \(e_0,e_1,e_2,\ldots\) and \(S\dd X \to X\) be
the shift---that is, \(S e_n = e_{n+1}\) (\(n \geqsl 0\)). Further, let \(\ssS\)
be the additive group of all natural numbers, \(\Phi\dd \ssS \ni n \mapsto S^n
\in \llL(X)\) and \(Y\) be the range of \(S\). Then all assumptions
of \THM{bdd}, except for ``\(S(Y) = Y\)'', are fulfilled and \(Y\) is
complemented. However, there is no projection onto \(Y\) that commutes with
\(S\).\par
The above example can also be employed to show that the assumption in \COR{bdd}
that all the operators \(B_s\) are isomorphisms (see (i1)) cannot be dropped in
general. Indeed, it suffices to set, in addition to the above settings, \(T_0\)
as the identity operator on \(E \df Y \subset X\), and \(\Psi = \Phi\).
\end{exm}

Further results of this section deal with renorming-like issues. The first
of them generalises results of Koehler and Rosenthal \cite{k-r} (see therein for
the definition of a semi-inner product on a Banach space).

\begin{thm}{renorm}
Let \(\ssS \ni s \mapsto T_s \in \llL(X)\) be a homomorphism of a weakly right
amenable SGT \(\ssS\) into \(\llL(X)\) \textup{(}where \((X,\|\cdot\|_X)\) is
a Banach space\textup{)} that is continuous in the strong operator topology of
\(\llL(X)\). Further, let \(m\) and \(M\) be two positive real constants such
that
\begin{equation}\label{eqn:aux4}
m \|x\|_X \leqsl \|T_s x\|_X \leqsl M \|x\|_X \qquad (x \in X,\ s \in \ssS).
\end{equation}
Then there exists a norm \(\|\cdot\|_*\) on \(X\) such that for any \(x \in X\)
and \(s \in \ssS\), \(m \|x\|_X \leqsl \|x\|_* \leqsl M \|x\|_X\) and \(\|T_s
x\|_* = \|x\|_*\).\par
If, in addition, the topology of \(\ssS\) is discrete, then there exists
a semi-inner product \([\cdot,-]_*\) inducing the norm \(\|\cdot\|_*\) such that
\([T_s x,T_s y]_* = [x,y]_*\) for all \(x, y \in X\) and \(s \in \ssS\).
\end{thm}
\begin{proof}
For any seminorm \(\|\cdot\|\) on \(X\) and \(s \in \ssS\) we denote by
\(\|\cdot\|.s\) the seminorm on \(X\) that assigns to a vector \(x\) the number
\(\|T_s x\|\). It is easy to check that
\begin{equation}\label{eqn:aux6}
(\|\cdot\|,s) \mapsto \|\cdot\|.s
\end{equation}
is an affine right action of \(\ssS\) on the set of all seminorms on \(X\).\par
Firstly, let \(\Delta\) be the set of all norms \(\|\cdot\|\) on \(X\) such
that \(m \|x\|_X \leqsl \|x\| \leqsl M \|x\|_X\) and \(m \|x\|_X \leqsl
\|T_s x\| \leqsl M \|x\|_X\) for all \(x \in X\) and \(s \in \ssS\). We equip
\(\Delta\) with the pointwise convergence topology. Observe that \(\Delta\) is
convex and compact (by the Tychonoff theorem), and \(\|\cdot\|_X.s \in \Delta\)
(by \eqref{eqn:aux4}) and \(\|\cdot\|.s \in \Delta\) for all \(\|\cdot\| \in
\Delta\) and \(s \in \ssS\). One also easily verifies that the action
\eqref{eqn:aux6} is jointly continuous on \(\Delta \times \ssS\). Hence, by
\THM{fpp}, there is a norm \(\|\cdot\|_* \in \Delta\) which is a fixed point for
this action. This means that all \(T_s\) are isometries with respect to this
norm.\par
Secondly, let \(\Gamma\) consist of all semi-inner products on \(X\) that induce
the norm \(\|\cdot\|_*\). As before, \(\Gamma\) is convex, compact (in
the pointwise convergence topology) and non-empty. For \([\cdot,-] \in \Gamma\)
and \(s \in \ssS\), we define \([\cdot,-].s \in \Gamma\) by \([x,y].s \df
[T_s x,T_s y]\). As before, one easily checks that in this way we have defined
an affine right action of \(\ssS\) on \(\Gamma\). So, if the topology of
\(\ssS\) is discrete, the fixed point property formulated in \THM{fpp} finishes
the proof.
\end{proof}

\begin{cor}{enlarge}
Let \(\ssS \ni s \mapsto T_s \in \llL(X)\) be a homomorphism of an SGT \(\ssS\)
into \(\llL(X)\) \textup{(}where \((X,\|\cdot\|)\) is a Banach space\textup{)}
that is continuous in the strong operator topology of \(\llL(X)\). Further,
assume \(\ssS_0\) is a weakly right amenable subsemigroup of \(\ssS\)
\textup{(}in the topology inherited from \(\ssS\)\textup{)} such that
\begin{itemize}
\item[\((\star)\)] \(\ssS\) is a unique subsemigroup \(\zzZ\) of \(\ssS\) that
 is closed, contains \(\ssS_0\) and has the following property: if \(x, y \in
 \ssS\) and \(x, xy \in \zzZ\), then \(y \in \zzZ\).
\end{itemize}
If \(m\) and \(M\) are two positive real constants such that
\begin{equation}\label{eqn:aux12}
m \|u\| \leqsl \|T_s u\| \leqsl M \|u\| \qquad (u \in X,\ s \in \ssS_0),
\end{equation}
then for any \(s \in \ssS\) and \(u \in X\),
\[\frac{m}{M} \|u\| \leqsl \|T_s u\| \leqsl \frac{M}{m} \|u\|.\]
\end{cor}
\begin{proof}
It follows from \THM{renorm} that there exists a norm \(\|\cdot\|_*\) on \(X\)
such that \(m \|u\| \leqsl \|u\|_* \leqsl M \|u\|\) and \(\|T_s u\|_* =
\|u\|_*\) (that is, \(T_s\) is isometric with respect to \(\|\cdot\|_*\)) for
any \(s \in \ssS_0\) and \(u \in X\). Now let \(\zzZ\) consist of all \(s \in
\ssS\) such that \(T_s\) is isometric with respect to \(\|\cdot\|_*\). We
conclude from \((\star)\) that \(\zzZ = \ssS\). But then, for any \(s \in \ssS\)
and \(u \in X\), \(\|T_s u\| \leqsl \frac1m \|T_s u\|_* = \frac1m \|u\|_* \leqsl
\frac{M}{m} \|u\|\) and similarly \(\|T_s u\| \geqsl \frac1M \|T_s u\|_* =
\frac1M \|u\|_* \geqsl \frac{m}{M} \|u\|\).
\end{proof}

\begin{rem}{alg}
The above result gives us some insight into the algebraic structure of a right
amenable SGT and tells us how far it is from being free. If, for example,
\(\ssS_0\) is a weakly right amenable SGT that generates a topological group
\(\ggG\) and \(\ggG \ni g \mapsto T_g \in \llL(X)\) is a unital homomorphism
that is continuous in the strong operator topology and satisfies
\eqref{eqn:aux12} (see \COR{enlarge}), then the operators of the form
\(T_{s_1}^{\epsi_1} \cdot \ldots \cdot T_{s_k}^{\epsi_k}\) (where \(s_1,\ldots,
s_k \in \ssS_0\), \(\epsi_1,\ldots,\epsi_k \in \{-1,1\}\) and \(k > 0\) are
arbitrary) are uniformly bounded.
\end{rem}

We end this section with the next two generalisations of the Sz.-Nagy theorem
\cite{SzN}. Below we use \(I\) to denote the identity operator on a Hilbert
space \(H\).

\begin{pro}{gen}
Let \(\ssS \ni s \mapsto T_s \in \llL(H)\) be a homomorphism of a weakly right
amenable SGT \(\ssS\) into \(\llL(H)\) \textup{(}where \((H,\scalarr)\) is
a Hilbert space\textup{)} that is continuous in the strong operator topology of
\(\llL(H)\). Further, suppose there exist an operator \(A_0 \in \llL(H)\) and
two functions \(m, M\dd H \to (0,\infty)\) such that for any \(x \in H\),
\begin{equation}\label{eqn:aux5}
m(x) \|x\| \leqsl \|A_0 T_s x\| \leqsl M(x) \|x\| \qquad (s \in \ssS).
\end{equation}
Then there exist a positive operator \(A \in \llL(H)\) with trivial kernel and
a homomorphism \(\ssS \ni s \mapsto V_s \in \llL(H)\) continuous in the strong
operator topology such that
\begin{itemize}
\item \(m(x) \|x\| \leqsl \|A x\| \leqsl M(x) \|x\|\) for any \(x \in H\);
\item \(A T_s = V_s A\) for all \(s \in \ssS\);
\item \(V_s\) is an isometry for any \(s \in \ssS\).
\end{itemize}
\end{pro}
\begin{proof}
It follows from \eqref{eqn:aux5} and the Uniform Boundedness Principle that
\[\sup_{s\in\ssS} \|A_0 T_s\| < \infty.\] Consequently, we may and do assume
that
\begin{equation}\label{eqn:Mbdd}
\sup_{x \in H} M(x) < \infty.
\end{equation}
For any \(B \in \llL(H)\) and \(s \in \ssS\) let \(B.s\) stand for the operator
\(T_s^* B T_s\). It is easily seen that \((B,s) \mapsto B.s\) is an affine right
action of \(\ssS\) on \(\llL(H)\).\par
Denote by \(\Delta\) the set of all positive operators \(B \in \llL(H)\) such
that
\begin{equation}\label{eqn:aux19}
m(x)^2 \|x\|^2 \leqsl \scalar{Bx}{x} \leqsl M(x)^2 \|x\|^2
\end{equation}
as well as
\begin{equation}\label{eqn:aux9}
m(x)^2 \|x\|^2 \leqsl \scalar{(B.s)x}{x} \leqsl M(x)^2 \|x\|^2
\end{equation}
for any \(x \in H\) and \(s \in \ssS\), and equip \(\Delta\) with the weak
operator topology. It is clear that \(\Delta\) is convex and compact, and that
\((A_0^* A_0).s \in \Delta\) (by \eqref{eqn:aux5}) and \(B.s \in \Delta\) for
any \(B \in \Delta\) and \(s \in \ssS\). We claim that the action just defined
is jointly continuous on \(\Delta \times \ssS\). To be convinced of that, assume
\((B_{\sigma})_{\sigma\in\Sigma} \subset \Delta\) and
\((s_{\sigma})_{\sigma\in\Sigma} \subset \ssS\) are two nets that converge to,
respectively, \(B \in \Delta\) and \(s \in \ssS\). Then \(T_{s_{\sigma}}\)
converge in the strong operator topology to \(T_s\). Since the operators
\(B_{\sigma}\) are uniformly bounded (by \eqref{eqn:Mbdd}, \eqref{eqn:aux19} and
\eqref{eqn:aux9})), we conclude that \(B_{\sigma} T_{s_{\sigma}}\) converge to
\(B T_s\) in the weak operator topology. So, when \(x, y \in H\) are fixed,
the net \((T_{s_{\sigma}} x)_{\sigma\in\Sigma}\) is eventually bounded and thus
\[\lim_{\sigma\in\Sigma} \scalar{T_{s_{\sigma}}^* B_{\sigma}
T_{s_{\sigma}} x}{y} = \lim_{\sigma\in\Sigma} \scalar{B_{\sigma} T_{s_{\sigma}}
x}{T_{s_{\sigma}} y} = \scalar{B T_s x}{T_s y} = \scalar{T_s^* B T_s x}{y}\]
(because the vectors \(B_{\sigma} T_{s_{\sigma}} x\) are uniformly bounded and
converge weakly to \(B T_s x\), and \(T_{s_{\sigma}} y\) converge in the norm
to \(T_s y\)).\par
Now the fixed point property from \THM{fpp} gives us an operator \(B \in
\Delta\) such that
\begin{equation}\label{eqn:aux11}
T_s^* B T_s = B \qquad (s \in \ssS).
\end{equation}
We define \(A\) as the (positive) square root of \(B\). For any \(x \in H\),
\(\|A x\|^2 = \scalar{Bx}{x}\) and thus \(m(x) \|x\| \leqsl \|A x\| \leqsl M(x)
\|x\|\) (thanks to \eqref{eqn:aux19}), which in turn implies that the kernel of
\(A\) is trivial.\par
Fix for a moment \(s \in \ssS\) and let \(V_s\) be the partial isometry that
appears in the polar decomposition \(A T_s = V_s |A T_s|\) (where \(|A T_s| =
\sqrt{(A T_s)^* (A T_s)}\)) of the operator \(A T_s\). The equation
\eqref{eqn:aux11} gives \((A T_s)^* (A T_s) = A^2\). So, by the uniqueness of
the positive square root we obtain \(|A T_s| = A\) and consequently \(A T_s =
V_s A\). Since \(A\) has trivial kernel, \(V_s\) is an isometry.\par
It remains to check that \(s \mapsto V_s\) is a homomorphism continuous in
the strong operator topology. For \(s, t \in \ssS\) we have \(V_{st} A = A
T_{st} = A T_s T_t = V_s A T_t = V_s V_t A\). Since the range of \(A\) is dense
in \(H\), we get \(V_{st} = V_s V_t\). Finally, if \(s_{\sigma} \in \ssS\)
converge to \(s \in \ssS\), then for any \(x \in H\), \(\lim_{\sigma\in\Sigma}
V_{s_{\sigma}}(Ax) = \lim_{\sigma\in\Sigma} A(T_{s_\sigma} x) = A T_s x =
V_s(Ax)\) and, again, the density of the range of \(A\) in \(H\) implies that
\(V_{s_{\sigma}}\) converge to \(V_s\) in the strong operator topology.
\end{proof}

As an immediate consequence of the above result, we obtain

\begin{cor}{iso}
Let \(\ssS \ni s \mapsto T_s \in \llL(H)\) be a homomorphism of a weakly right
amenable SGT \(\ssS\) into \(\llL(H)\) \textup{(}where \(H\) is a Hilbert
space\textup{)} that is continuous in the strong operator topology of
\(\llL(H)\). There exists an invertible positive operator \(A \in \llL(H)\) such
that \(A T_s A^{-1}\) is an isometry for each \(s \in \ssS\) iff there are two
positive real constants \(m\) and \(M\) such that
\begin{equation}\label{eqn:aux10}
m \|x\| \leqsl \|T_s x\| \leqsl M \|x\| \qquad (s \in \ssS,\ x \in H).
\end{equation}
Moreover, if \eqref{eqn:aux10} holds, the above operator \(A\) can be chosen so
that \(m I \leqsl A \leqsl M I\).
\end{cor}

The proof is left to the reader.

\begin{rem}{SzN}
\COR{iso} combined with \COR{enlarge} leads to a stronger generalisation of
the Sz.-Nagy theorem that reads as follows:
\begin{quote}\itshape
If \(\ssS \ni s \mapsto T_s \in \llL(H)\) is a homomorphism of an SGT \(\ssS\)
into \(\llL(H)\) \textup{(}where \(H\) is a Hilbert space\textup{)} that is
continuous in the strong operator topology and there are a weakly right amenable
subsemigroup \(\ssS_0\) of \(\ssS\) and two positive real constants \(m\) and
\(M\) such that condition \((\star)\) of \COR{enlarge} is fulfilled and
the inequality \(m \|x\| \leqsl \|T_s x\| \leqsl M \|x\|\) holds for any \(s \in
\pmb{\ssS}_{\pmb0}\) and \(x \in H\), then there exists an invertible positive
operator \(A \in \llL(H)\) such that \(A T_s A^{-1}\) is an isometry for each
\(s \in \ssS\).
\end{quote}
In particular, if \(\ggG\) is a topological group and a weakly right amenable
semigroup \(\ssS \subset \ggG\) generates \(\ggG\), then each unital
homomorphism \(\Phi\dd \ggG \ni g \mapsto T_g \in \llL(H)\) that is continuous
on \(\ssS\) in the strong operator topology and satisfies condition
\eqref{eqn:aux10} of \COR{iso} has image \(\Phi(\ggG)\) ``similar'' to
a subgroup of the unitary group of \(H\).
\end{rem}

\begin{proof}%
[Proofs of {\THM[s]{m1}}, \ref{thm:m2}, \ref{thm:m3} and \ref{thm:m4}]
As all discrete abelian semigroups are strongly right amenable (see item (B) in
\EXS{amen}), the theorems of the first section are special cases of the results
of this part.
\end{proof}

\section{More on amenability}

\begin{proof}[Proof of \THM{fpp}]
Let \(\FfF \df C_{norm}(\ssS)\) [resp.\ \(\FfF \df C_{weak}(\ssS)\)].\par
First assume (SG2) (or (SG2')) holds. Let the set \(\Delta\) of all means on
\(\FfF\) be equipped with the weak* topology. For \(\phi \in \Delta\) and \(s
\in \ssS\) we define \(\phi.s \in \Delta\) by \((\phi.s)(f) = \phi(f_s)\). Then
\((\phi,s) \mapsto \phi.s\) is a (unital if \(\ssS\) is unital) affine right
action that is jointly [resp.\ separately] continuous. So, an application of
(SG2) (or (SG2')) yields the existence of a right invariant mean.\par
Now assume \(\ssS\) is right amenable and let \(K\) and \(K \times \ssS \ni
(x,s) \mapsto x.s \in K\) be, respectively, a non-empty compact convex set in
a Hausdorff locally convex topological vector space and a jointly [resp.\
separately] continuous affine right action. Denote by \(CA(K) \subset C(K)\) and
\(j\), respectively, the Banach space of all real-valued continuous affine
functions on \(K\) and the function from \(CA(K)\) constantly equal to \(1\).
Fix a right invariant mean \(\phi\) on \(\FfF\) and a point \(b \in K\). For
each \(u \in CA(K)\) let \(\hat{u} \in \FfF\) be given by \(\hat{u}(s) =
u(b.s)\). Observe that \(\Lambda\dd CA(K) \ni u \mapsto \phi(\hat{u}) \in \RRR\)
is a linear functional such that \(\Lambda(j) = 1 = \|\Lambda\|\). It follows
from a well-known characterisation of such functionals on \(CA(K)\) that there
is \(a \in K\) for which \(\Lambda(u) = u(a)\) for any \(u \in CA(K)\). In order
to check that \(a\) is a fixed point for the action, it suffices to verify that
\(u(a.s) = u(a)\) for any \(u \in CA(K)\). To this end, note that \(\hat{u}_s =
\widehat{u_s}\) for any such a function \(u\) and all \(s \in \ssS\). But then
\(u(a.s) = u_s(a) = \Lambda(u_s) = \phi(\widehat{u_s}) = \phi(\hat{u}_s) =
\phi(\hat{u}) = \Lambda(u) = u(a)\) and we are done.
\end{proof}

To formulate our next result, let us introduce some elementary notion. For any
(unital or non-unital) SGT \(\ssS\) we use \(\unit{\ssS}\) to denote its
\textit{unitization}; that is, \(\unit{\ssS}\) is a unital SGT such that
\(\unit{\ssS} = \ssS \sqcup \{1\}\) where \(1 \notin \ssS\) is the neutral
element of \(\unit{\ssS}\) as well as its isolated point, the multiplication
of \(\unit{\ssS}\) extends that in \(\ssS\), \(\ssS\) is both closed and open in
\(\unit{\ssS}\) and the topology of \(\ssS\) coincides with the one inherited
from \(\unit{\ssS}\). Note that even if \(\ssS\) is unital, the neutral element
of \(\unit{\ssS}\) differs from the one of \(\ssS\).\par
Below we list basic properties of the class of right amenable SGT's.

\begin{pro}{pro}
\begin{enumerate}[\upshape(A)]
\item An SGT \(\ssS\) is weakly \textup{[}resp.\ strongly\textup{]} right
 amenable iff so is \(\unit{\ssS}\).
\item Let \(h\dd \ttT \to \ssS\) be a continuous onto homomorphism between
 SGT's. If \(\ttT\) is weakly \textup{[}resp.\ strongly\textup{]} right
 amenable, so is \(\ssS\).
\item Let \(\ttT\) be a subsemigroup of an SGT \(\ssS\) such that \(\ssS\) is
 a unique subsemigroup \(\zzZ\) of \(\ssS\) that is closed, contains \(\ttT\)
 and has the following property:
 \begin{itemize}
 \item[\((\star\star)\)] if \(x, y \in \ssS\) and \(x, xy \in \zzZ\), then \(y
  \in \zzZ\).
 \end{itemize}
 Then, if \(\ttT\) is weakly \textup{[}resp.\ strongly\textup{]} right amenable,
 so is \(\ssS\).
\item Let \(\{\ssS_b\}_{b \in B}\) be a family of subsemigroups of an SGT
 \(\ssS\) such that \(\ssS\) coincides with the smallest closed subsemigroup
 that contains \(\bigcup_{b \in B} \ssS_b\). Assume \(s s' = s' s\) for any \(s
 \in \ssS_b\) and \(s' \in \ssS_{b'}\) \textup{(}\(b, b' \in B\)\textup{)} with
 \(b \neq b'\). If all \(\ssS_b\) are weakly \textup{[}resp.\ strongly\textup{]}
 right amenable, so is \(\ssS\).
\item The \textup{(}topological\textup{)} product of a family of weakly
 \textup{[}resp.\ strongly\textup{]} right amenable \textbf{unital} SGT's is
 weakly \textup{[}resp.\ strongly\textup{]} right amenable as well.
\end{enumerate}
\end{pro}

We do not know whether the assumption in (E) that SGT's are unital can be
dropped (even when the family of SGT's is finite).

\begin{proof}[Proof of \PRO{pro}]
Item (A) is left to the reader as a simple exercise. To prove each of (B), (C)
and (D), fix a jointly [resp.\ separately] continuous affine right action
\begin{equation}\label{eqn:aux17}
K \times \ssS \ni (x,s) \mapsto x.s \in K
\end{equation}
of \(\ssS\) on a non-empty compact convex set \(K\). Under the settings of (B),
also \(K \times \ttT \ni (x,t) \mapsto x.h(t) \in K\) is such an action. So,
\THM{fpp} implies that this last action has a fixed point. It follows from
the surjectivity of \(h\) that also \eqref{eqn:aux17} has a fixed point, and we
are done. Under the settings of (C), there is \(a \in K\) such that \(a.t = a\)
for any \(t \in \ttT\). Since \(\zzZ \df \{s \in \ssS\dd\ a.s = a\}\) is
a closed subsemigroup of \(\ssS\) that contains \(\ttT\) and satisfies
\((\star\star)\), the assumptions of (C) imply that \(\ssS = \zzZ\), which means
that, again, \eqref{eqn:aux17} has a fixed point.\par
We turn to (D). For any \(b \in B\) set \(F_b \df \{x \in K\dd\ x.s = x
\textup{ for any } s \in \ssS_b\}\). It follows that \(F_b\) is compact, convex,
and non-empty---by a suitable amenability of \(\ssS_b\). We claim that \(L \df
\bigcap_{b \in B} F_b\) is non-empty as well. To be convinced of that, it is
enough to check that \(\bigcap_{k=1}^n F_{b_k} \neq \varempty\) for any \(b_1,
\ldots,b_n \in B\) (and \(n > 0\)). We proceed by induction on \(n\). The case
\(n = 1\) has already been established. Now if \(n > 1\) and \(b_1,\ldots,b_n
\in B\) are arbitrary, then \(D \df \bigcap_{k=1}^{n-1} F_{b_k}\) is compact,
convex, and non-empty (by the induction hypthesis). Moreover, if \(s \in
\ssS_{b_n}\) and \(x \in D\), then \(x.s \in D\). (Indeed, if \(k < n\) and \(t
\in \ssS_{b_k}\), then \((x.s).t = x.(st) = x.(ts) = (x.t).s = x.s\).) So, \(D
\times \ssS_{b_n} \ni (x,s) \mapsto x.s \in D\) is a correctly defined action
that has a fixed point, by \THM{fpp}. Equivalently, \(D \cap F_{b_n} \neq
\varempty\), which finishes the proof that \(L\) is non-empty. Now let \(a \in
L\). The set \(\ttT \df \{s \in \ssS\dd\ a.s = a\}\) is a closed subsemigroup of
\(\ssS\) that contains \(\bigcup_{b \in B} \ssS_b\) and hence \(\ttT = \ssS\),
thanks to the assumptions of (D).\par
Finally, we turn to (E). So, assume \(\ssS\) is the product of suitably amenable
unital SGT's \(\ttT_b\) (\(b \in B\)) and for all \(b \in B\) denote by \(e_b\)
the neutral element of \(\ttT_b\). Further, for any \(d \in B\) and \(t \in
\ttT_d\) denote by \(t^{(d)}\) the element \((z_b)_{b \in B} \in \ssS\) such
that \(z_b = e_b\) for \(b \neq d\) and \(z_d = t\), and set \(\ssS_d \df
\{x^{(d)}\dd\ x \in \ttT_d\}\). Since \(\ttT_d \ni x \mapsto x^{(d)} \in
\ssS_d\) is a continuous surjective homomorphism, it follows from item (B) that
\(\ssS_d\) is suitably amenable. Now it suffices to observe that for
\(\{\ssS_b\}_{b \in B}\) and \(\ssS\) all assumptions of item (D) are satisfied
and apply this part.
\end{proof}

The next property (in the class of topological groups), with which we end
the paper, is less elementary.

\begin{thm}{dense}
If an SGT \(\ssS\) contains a dense subsemigroup that is weakly \textup{[}resp.\
strongly\textup{]} right amenable, then \(\ssS\) itself is weakly
\textup{[}resp.\ strongly\textup{]} right amenable.\par
A dense subgroup of a weakly right amenable \textbf{topological group} is weakly
right amenable as well.
\end{thm}
\begin{proof}
The first claim of the theorem is a special case of item (C) of \PRO{pro}.\par
To prove the second claim, we will make use of so-called Ra\u{\i}kov-complete
topological groups (see, e.g., Section~3.6 in \cite{a-t}). We need the following
two facts about them:
\begin{enumerate}[(RC1)]
\item If \(u\dd H_0 \to G\) is a continuous homomorphism of a dense subgroup
 \(H_0\) of a topological group \(H\) into a Ra\u{\i}kov-complete topological
 group \(G\), then \(u\) (uniquely) extends to a continuous homomorphism from
 \(H\) into \(G\).
\item The homeomorphism group (equipped with the compact-open topology) of
 a compact Hausdorff space is Ra\u{\i}kov-complete.
\end{enumerate}
Now assume \(H_0\) is a dense subgroup of a weakly right amenable topological
group \(H\). Let \(K \times H_0 \ni (x,h) \mapsto x.h \in K\) be a (jointly)
continuous affine right action of \(H_0\) on a non-empty compact convex set
\(K\). For any \(h \in H_0\) denote by \(f_h\dd K \to K\) the homeomorphism
\(x \mapsto x.h^{-1}\). It is clear that \(\Phi_0\dd h \mapsto f_h\) is
a homomorphism of \(H_0\) into the homeomorphism group \(G\) of \(K\). What is
more, the joint continuity of the action is equivalent to the continuity of this
homomorphism. So, (RC1) combined with (RC2) implies that \(\Phi_0\) extends to
a continuous homomorphism \(\Phi\dd H \to G\). We conclude that the given action
(of \(H_0\) on \(K\)) extends to a jointly continuous right action \((x,h)
\mapsto \Phi(h^{-1})(x)\) of \(H\) on \(K\). It is easy to verify that this
action is affine and hence it has a fixed point, which finishes the proof.
\end{proof}


\begin{thebibliography}{28}

\bibitem{aki} E.\ Akhiezer,
 \textit{Recurrence in Topological Dynamics. Furstenberg Families and Ellis Actions},
 Springer Science+Business Media, New York, 1997.

\bibitem{a-t} A.V.\ Arhangel'skii and M.G.\ Tkachenko,
 \textit{Topological Groups and Related Structures},
 Atlantis Press, Paris; World Scientific, Hackensack, NJ, 2008.

\bibitem{c-m} E.W.\ Cheney, P.D.\ Morris,
 \textit{On the existence and characterization of minimal projections},
 J.\ Reine Angew.\ Math.\ \textbf{270} (1974), 61--76.

\bibitem{day} M.M.\ Day,
 \textit{Fixed point theorems for compact convex sets},
 Illinois J.\ Math.\ \textbf{5} (1961), 585--590.

\bibitem{dy2} M.M.\ Day,
 \textit{Correction to my paper ``Fixed point theorems for compact convex sets''},
 Illinois J.\ Math.\ \textbf{8} (1964), 713.

\bibitem{f-s} I.\ Farah and S.\ Solecki,
 \textit{Extreme amenability of $L_0$, a Ramsey theorem, and L\'{e}vy groups},
 J.\ Funct.\ Anal.\ \textbf{255} (2008), 471--493.

\bibitem{gre} F.P.\ Greenleaf,
 \textit{Invariant Means on Topological Groups and Their Applications},
 Van Nostrand, New York, 1969.

\bibitem{g-m} M.\ Gromov and M.D.\ Milman,
 \textit{A topological application of the isoperimetric inequality},
 Amer.\ J.\ Math.\ \textbf{105} (1983), 843--854.

\bibitem{jhp} M.\ de Jeu, R. El Harti and P.R.\ Pinto,
 \textit{Amenable crossed product Banach algebras associated with a class of \(C^*\)-dynamical systems},
 Integral Equations Operator Theory \textbf{87} (2017), 169--178.

\bibitem{d-s} N.\ Dunford and J.T.\ Schwartz,
 \textit{Linear Operators.\ Part I: General Theory},
 Interscience Publishers, New York, 1958.

\bibitem{kak} S.\ Kakutani,
 \textit{Two fixed\hyp{}point theorems concerning bicompact convex sets},
 Proc.\ Imp.\ Acad.\ Tokyo \textbf{14} (1938), 242--245.

\bibitem{kpt} A.S.\ Kechris, V.G.\ Pestov and S.\ Todorcevic,
 \textit{Fra\"{\i}ss\'{e} limits, Ramsey theory, and topological dynamics of automorphism groups},
 Geom.\ \& Funct.\ Anal.\ \textbf{15} (2005), 106--189.

\bibitem{k-r} D.\ Koehler and P.\ Rosenthal,
 \textit{On isometries of normed linear spaces},
 Studia Math.\ \textbf{36} (1970), 213--216.

\bibitem{lin} H.\ Lin,
 \textit{Residually finite\hyp{}dimensional and AF\hyp{}embeddable $C^*$\hyp{}algebras},
 Proc.\ Amer.\ Math.\ Soc.\ \textbf{129} (2001), 1689--1696.

\bibitem{mar} A.\ Markov,
 \textit{Quelques th\'{e}oremes sur les ensembles ab\'{e}liens},
 Dokl.\ Akad.\ Nauk SSSR \textbf{10} (1936), 311--314.

\bibitem{m-t} J.\ Melleray and T.\ Tsankov,
 \textit{Generic representations of abelian groups and extreme amenability},
 Israel J.\ Math.\ \textbf{198} (2013), 129--167.

\bibitem{pat} A.L.T.\ Paterson,
 \textit{Amenability}
 (Mathematical Surveys and Monographs, Vol.\ 29),
 American Mathematical Society, Providence, Rhode Island, 1988.

\bibitem{pie} J.-P.\ Pier,
 \textit{Amenable Locally Compact Groups},
 Wiley\hyp{}Interscience, New York, 1984.

\bibitem{r-s} C.\ Rosendal and S.\ Solecki,
 \textit{Automatic continuity of homomorphisms and fixed points on metric compacta},
 Israel J.\ Math.\ \textbf{162} (2007), 349--371.

\bibitem{rud} W.\ Rudin,
 \textit{Projections on invariant subspaces},
 Proc.\ Amer.\ Math.\ Soc.\ \textbf{13} (1962), 429--432.

\bibitem{rd2} W.\ Rudin,
 \textit{Functional Analysis},
 McGraw\hyp{}Hill, Inc., New York, 1991.

\bibitem{run} V.\ Runde,
 \textit{Lectures on Amenability}
 (Lecture Notes in Mathematics, vol. 1774),
 Springer-Verlag, Berlin, 2002.

\bibitem{sab} M.\ Sabok,
 \textit{Extreme amenability of abelian $L_0$ groups},
 J.\ Funct.\ Anal.\ \textbf{263} (2012), 2978--2992.

\bibitem{SzN} B.\ Sz.\hyp{}Nagy,
 \textit{On uniformly bounded linear transformations in Hilbert space},
 Acta Sci.\ Math.\ (Szeged) \textbf{11} (1947), 152--157.

\end{thebibliography}
\end{document}